\documentclass[11pt]{article}
\usepackage[utf8]{inputenc}
\usepackage[T1]{fontenc}
\usepackage{amsmath,amssymb,amsfonts}
\usepackage{amsthm}
\usepackage{graphicx}
\usepackage{booktabs}
\usepackage{geometry}
\usepackage{hyperref}
\hypersetup{colorlinks=true, linkcolor=blue, citecolor=blue, urlcolor=blue}
\usepackage{float}
\usepackage{caption}

\newtheorem{theorem}{Theorem}
\newtheorem{definition}{Definition}

\newtheorem{corollary}{Corollary}
\theoremstyle{remark}
\newtheorem{remark}{Remark}

\title{Bias-Corrected Estimators for Joint Entropy,\\
Conditional Entropy, and Mutual Information\\
in the Discrete Bivariate Case:\\
Theory and an Application to Motor Insurance} 

\author{Amadou Diadi\'{e} B\^a\\
LERSTAD, Gaston Berger University, Saint-Louis, S\'{e}n\'{e}gal\\
\texttt{ba.amadou-diadie@ugb.edu.sn} \ \ \texttt{ba1amadou@yahoo.fr}}

\date{}

\begin{document}
\maketitle

\begin{abstract}
This paper studies the finite-sample bias of plug-in estimators for joint 
entropy, conditional entropy, and mutual information for finitely supported 
discrete random variables. Using the univariate reduction framework developed 
in our companion works \cite{Ba2019a,Ba2020}, we derive explicit first-order bias formulas. For a 
pair $(X,Y)$ with support sizes $r$ and $s$, the biases are respectively
\[
-\frac{rs-1}{2n},\qquad
-\frac{r(s-1)}{2n},\qquad
\frac{(r-1)(s-1)}{2n}.
\]
We then propose bias-corrected estimators for the three measures and prove 
that the uncorrected and corrected versions are asymptotically equivalent, 
differing by a deterministic term of order $O(1/n)$; consequently, they share 
the same asymptotic distribution. A simulation study validates the theoretical 
results, including a dedicated independence study and the relationship with 
Wilks' theorem. An application to a motor insurance portfolio, where $X$ is 
the driver's age class and $Y$ the claim severity class, illustrates the 
improvement brought by the bias correction for risk classification.\\


\end{abstract}

\section{Introduction}

\subsection{Motivation}
In many scientific applications, understanding the relationship between two 
random variables is essential. In neuroscience, mutual information is used 
to quantify information transmission in neural spike trains \cite{Paninski2003}. 
In genomics, joint entropy and mutual information are used to quantify genetic 
diversity and to identify conserved regions \cite{Shenkin1991}. In finance and 
insurance, mutual information between risk factors is of paramount importance 
for risk assessment, pricing, and solvency analysis \cite{Schreiber2000}.

For a pair of discrete random variables, the natural information-theoretic 
measures are the joint entropy, the conditional entropy, and the mutual 
information. They quantify respectively the total uncertainty of the pair, the 
remaining uncertainty of one variable given the other, and the information 
shared between them. These measures extend the univariate Shannon entropy 
\cite{Shannon1948,Cover2006} and are central to information theory, 
statistics, and machine learning.

In practice, the true joint probability distribution (\textit{p.d.}) is unknown, and these 
measures must be estimated from observed data. The most natural approach is 
the plug-in estimator, obtained by replacing the unknown probabilities by 
their empirical counterparts. These estimators are intuitive, consistent, and 
asymptotically normal \cite{Ba2019a,Ba2020}.

However, plug-in estimators suffer from bias in finite samples. This bias can 
be substantial when the sample size $n$ is moderate or when the number of 
categories is large relative to $n$. The presence of bias can affect the 
accuracy of entropy estimates, particularly in applications where sample 
sizes are limited \cite{Miller1955,Basharin1959,Harris1975}.

\subsection{Applications}

\paragraph{Information theory and communication.}
Mutual information quantifies the amount of information that one random 
variable carries about another, serving as a measure of dependence and 
channel capacity \cite{Shannon1948,Cover2006}. Joint entropy measures the 
total uncertainty of a pair, while conditional entropy measures the remaining 
uncertainty after conditioning.

\paragraph{Statistics and machine learning.}
Mutual information is used for feature selection, independence testing, and 
divergence analysis \cite{Cover2006}. Conditional entropy plays a role 
in causal inference and model selection. These measures are widely used in 
unsupervised learning, clustering, and classification.

\paragraph{Insurance and actuarial science.}
In non-life insurance, risk classification relies on the analysis of 
dependence between rating factors (age, vehicle type, geographic zone) and 
claim characteristics (frequency, severity). Mutual information provides a 
natural measure of the predictive power of a rating factor, while conditional 
entropy quantifies the residual uncertainty on the claim severity given the 
rating factor. These quantities are relevant to pricing, segmentation, and 
risk aggregation.

\subsection{Previous Work and Motivation for Bias Correction}
The estimation of entropy from finite samples has a long history. 
The plug-in estimator, while intuitive and consistent, has been known to 
suffer from bias since the early works of Miller \cite{Miller1955}, 
Basharin \cite{Basharin1959}, and Harris \cite{Harris1975}. For Shannon 
entropy with $r$ categories, they established that
\[
\mathbb{E}\bigl[\widehat{H}^{(n)}\bigr] = H - \frac{r-1}{2n} + O(n^{-2}).
\]

For the bivariate case, the estimation of joint Shannon, Rényi, and Tsallis 
entropies, as well as mutual information, was studied in \cite{Ba2019a}, where 
almost sure rates of convergence and asymptotic normality were established 
via a univariate reduction principle. A related work is that of 
Antos and Kontoyiannis \cite{Antos2001}, who studied convergence properties 
of functional estimates for discrete distributions.

Beyond the plug-in estimator, bias corrections have been developed for a wide 
range of information measures: generalized entropies (Rényi, Tsallis), 
extropies, Kullback-Leibler divergence, and mutual information in continuous 
settings (KSG, BI-KSG, G-knn estimators). Bias corrections have also been developed for extropies, both in the 
univariate discrete case \cite{Ba2026extropy} and in related settings. For mutual information, several 
bias correction methods have been proposed, and it is known that the plug-in 
estimator has a positive bias of order $O(1/n)$ 
\cite{Paninski2003,Nemenman2004}. However, explicit first-order bias formulas 
for joint entropy and conditional entropy, derived from a unified framework, 
are still lacking. The present paper fills this gap by providing a unified 
framework for the bivariate discrete case, covering joint entropy, conditional 
entropy, and mutual information simultaneously, with explicit first-order bias 
formulas derived from the same univariate reduction principle.

\subsection{Main Contribution}
The main contributions of this paper are as follows:
\begin{enumerate}
\item \emph{Unified bias correction framework:} extension of the univariate 
reduction principle to the bivariate case, covering the joint entropy, the 
conditional entropy, and the mutual information.
\item \emph{Explicit bias formulas:} derivation of explicit first-order bias 
formulas for the joint entropy, the conditional entropy, and the mutual 
information.
\item \emph{Bias-corrected estimators:} proposal of an explicit bias-corrected 
estimator for each measure.
\item \emph{Asymptotic equivalence:} proof that the uncorrected and 
bias-corrected estimators differ by a term of order $O(1/n)$ and are therefore 
asymptotically equivalent. This is the central result of the paper.
\item \emph{Independence and bias correction:} a dedicated study of the 
behaviour of mutual information under independence, showing that the plug-in 
estimator overestimates information and that the bias-corrected 
estimator is centred near zero.
\item \emph{Link with Wilks' theorem:} a connection between the 
positive bias of the plug-in mutual information estimator under independence 
and the mean of the limiting chi-square distribution appearing in Wilks' 
theorem.
\item \emph{Actuarial application:} an application to a motor insurance 
portfolio with three age classes and four claim severity classes, 
illustrating the practical relevance of the bias correction for risk 
classification.
\item \emph{Simulation validation:} a simulation study 
validating the theoretical findings, including a dedicated independence study 
and a chi-square goodness-of-fit study.
\end{enumerate}

\subsection{Paper Organization}
The paper is organized as follows. Section~\ref{sec:univ} presents the 
univariate reduction framework for two variables, defines the three measures, 
and sets up the estimation framework. Section~\ref{sec:bias} derives the 
bias-corrected estimators. Section~\ref{sec:asymp} establishes the asymptotic 
equivalence, the relationship with independence, and the connection with 
Wilks' theorem. Section~\ref{sec:sim} presents the simulation study, including 
the motor insurance application. Section~\ref{sec:concl} concludes.

\section{Univariate Reduction Framework}\label{sec:univ}

In the sequel, all the random variables (\textit{r.v.}'s) are defined on the 
same probability space $(\Omega,\mathcal{A},\mathbb{P})$. Additionally, 
entropies will be considered as functions of probability mass functions 
(\textit{p.m.f.'}s), since they only take into account probabilities of 
observing specific events.

\subsection{The Reduction Principle for Two Variables}
Let \(X,Y\) be two jointly distributed discrete \textit{r.v.}'s with respective supports 
\(X(\Omega)=\{x_{1},\dots,x_{r}\}\) and \(Y(\Omega)=\{y_{1},\dots,y_{s}\}\), 
(\(r,s\geq 2\)), and with joint \textit{p.d.} 
$\mathbf{p}_{(X,Y)}=(p_{i,j})_{(i,j)\in I\times J}$ where 
\[
p_{i,j}=\mathbb{P}(X=x_{i},Y=y_{j}),\quad (i,j)\in I\times J,
\]
with \(I=\{1,\dots,r\}\) and \(J=\{1,\dots,s\}\).\\

\noindent Following the univariate reduction framework developed in 
\cite{Ba2019a,Ba2020}, we define an auxiliary univariate \textit{r.v.}\ 
\(Z\) with support \(Z(\Omega)=\{z_{1},\dots,z_{rs}\}\) and \textit{p.d.} 
$\mathbf{p}_{Z}=(p_{Z,k})_{k\in K}$, by
\begin{equation}
p_{Z,k}= \mathbb{P}\!\left(
X = x_{1+\lfloor (k-1)/s\rfloor},\;
Y = y_{k-s\lfloor (k-1)/s\rfloor}
\right),
\label{eq:inverse}
\end{equation}
for all \(k\in K=\{1,\dots,rs\}\), where \(\lfloor x\rfloor\) denotes the 
largest integer less than or equal to \(x\).\\

\noindent Conversely, for any \textit{p.d.}\ $\mathbf{p}_{Z}=(p_{Z,k})_{k\in K}$ 
on \(Z(\Omega)\), we associate the ordered pair \((X,Y)\) with respective 
supports \(X(\Omega)=\{x_{1},\dots,x_{r}\}\) and \(Y(\Omega)=\{y_{1},\dots,y_{s}\}\), 
satisfying
\begin{equation}
p_{i,j}=\mathbb{P}\bigl(Z=z_{\delta_{i}^{j}}\bigr),
\quad \forall (i,j)\in I\times J,
\label{eq:forward}
\end{equation}
where
\begin{equation}
\delta_{i}^{j}=s(i-1)+j.
\label{eq:delta}
\end{equation}

\noindent It is important to note that the index map
\[
\delta:I\times J\to K,\qquad (i,j)\mapsto s(i-1)+j,
\]
has inverse
\[
\delta^{-1}:K\to I\times J,\qquad
k\mapsto \left(1+\Bigl\lfloor\frac{k-1}{s}\Bigr\rfloor,\,
k-s\Bigl\lfloor\frac{k-1}{s}\Bigr\rfloor\right).
\]

\subsection{Entropy Equivalence and Notational Simplification}\label{sec:equiv}
It follows from the univariate reduction framework that the joint entropy of 
the pair \((X,Y)\) coincides with the Shannon entropy of the univariate 
auxiliary variable \(Z\):
\begin{equation}
H(\mathbf{p}_{(X,Y)})=H(\mathbf{p}_Z)
=-\sum_{k=1}^{rs}p_{Z,k}\log p_{Z,k}.
\label{eq:HXY=HZ}
\end{equation}
Indeed, the bijection \(\delta\) implies that the two sums run over the same 
set of probabilities.

\noindent The marginals \textit{p.d.}'s $\mathbf{p}_X=\left(p_{X,i}\right)_{i\in I}$ and $\mathbf{p}_Y=\left(p_{Y,j}\right)_{j\in J}$
 are expressed in terms of 
\(\mathbf{p}_Z\) through:
\[
p_{X,i}=\sum_{j=1}^{s}p_{Z,\delta_{i}^{j}},
\qquad
p_{Y,j}=\sum_{i=1}^{r}p_{Z,\delta_{i}^{j}}.
\]

\noindent Consequently, the conditional entropy of \(Y\) given \(X\) can be 
written as a difference of univariate entropies:
\begin{equation}
H(\mathbf{p}_{Y|X})=H(\mathbf{p}_{Z})-H(\mathbf{p}_{X}),
\label{eq:cond}
\end{equation}
where \(H(\mathbf{p}_{X})=-\sum_{i=1}^{r}p_{X,i}\log p_{X,i}\).

\noindent Similarly, the mutual information between \(X\) and \(Y\) becomes
\begin{equation}
\mathcal{I}(\mathbf{p}_{X:Y})=H(\mathbf{p}_{X})+H(\mathbf{p}_{Y})-H(\mathbf{p}_{Z}).
\label{eq:mi}
\end{equation}

\noindent All joint and conditional entropies, as well as mutual information, 
can therefore be expressed as differences of univariate entropies. In the 
sequel, we use these univariate expressions when deriving bias 
corrections, which simplifies the proofs.

\subsection{Definitions of Bivariate Entropy Measures}

\begin{definition}
\label{def:entropies}
\mbox{}
\begin{enumerate}
\item[(a)] The \emph{joint entropy} of the pair \((X,Y)\) is defined as
\[
H(\mathbf{p}_{(X,Y)})=-\sum_{(i,j)\in I\times J}p_{i,j}\log p_{i,j}.
\]
\item[(b)] The \emph{conditional entropy of \(Y\) given \(X\)} is defined as
\[
H(\mathbf{p}_{Y|X})
=-\sum_{(i,j)\in I\times J}p_{i,j}\log\frac{p_{i,j}}{p_{X,i}}.
\]
\item[(c)] The \emph{mutual information between \(X\) and \(Y\)} is defined as
\[
\mathcal{I}(\mathbf{p}_{X:Y})
=\sum_{(i,j)\in I\times J}p_{i,j}\log\frac{p_{i,j}}{p_{X,i}\,p_{Y,j}}.
\]
\end{enumerate}
where the marginal masses are
\[
p_{X,i}=\sum_{j\in J}p_{i,j},
\qquad
p_{Y,j}=\sum_{i\in I}p_{i,j}.
\]
\end{definition}

\noindent We have the following fundamental relationships
\[
H(\mathbf{p}_{Y|X})=H(\mathbf{p}_{(X,Y)})-H(\mathbf{p}_{X}),
\]
\[
\mathcal{I}(\mathbf{p}_{X:Y})=H(\mathbf{p}_{X})+H(\mathbf{p}_{Y})-H(\mathbf{p}_{(X,Y)}).
\]

\noindent The quantity \(H(\mathbf{p}_{(X,Y)})\) is the average randomness or 
uncertainty of the jointly distributed \textit{r.v.'s} \(X,Y\), with high 
values reflecting a high degree of uncertainty. \(H(\mathbf{p}_{Y|X})\) 
represents the average randomness of \(Y\) given knowledge of \(X\), and 
\(\mathcal{I}(\mathbf{p}_{X:Y})\) quantifies the amount of information that 
\(X\) provides about \(Y\) (or vice versa).

\subsection{Estimation}\label{sec:estim}

In practice, the joint distribution 
$\mathbf{p}_{(X,Y)}=(p_{i,j})_{(i,j)\in I\times J}$ is unknown, and so are the 
marginal \textit{p.d.}'s $\mathbf{p}_X=\left( p_{X,i}\right)_{i\in I}$ and   $\mathbf{p}_Y=\left( p_{Y,j}\right)_{j\in J}$, on which the three bivariate measures 
depend. \\
To estimate them, we use the univariate reduction principle of 
Section~\ref{sec:univ}. \\
Let $Z_1,\dots,Z_n$  be an \textit{i.i.d.} sample of size $n$
from the \textit{p.d.} $\mathbf{p}_Z=\left(p_{Z,k}\right)_{k\in K}$. \\
 Define the empirical \textit{p.d.} of  $\mathbf{p}_Z$ as
\begin{equation}
\widehat{\mathbf{p}}_{Z}^{(n)}
= \left(\widehat{p}_{Z,k}^{(n)}\right)_{k\in K}\qquad \text{where}\qquad \widehat{p}_{Z,k}^{(n)}
= \frac{1}{n}\sum_{\ell=1}^{n}\mathbf{1}_{z_k}(Z_\ell),
\end{equation}
with $\mathbf{1}_{z_k}(Z_\ell)=1$ if $Z_\ell=z_k$ and $0$ otherwise.\\
 Using the reduction \eqref{eq:forward}, the corresponding empirical 
\textit{p.d.} of the joint \textit{p.d.} $\mathbf{p}_{(X,Y)}$ is
\begin{equation}
\widehat{\mathbf{p}}_{(X,Y)}^{(n)}
= \left(\widehat{p}_{i,j}^{(n)}\right)_{(i,j)\in I\times J}\qquad \text{where}\qquad \widehat{p}_{i,j}^{(n)}
= \frac{1}{n}\sum_{\ell=1}^{n}\mathbf{1}_{z_{\delta_{i}^{j}}}(Z_{\ell})
\label{eq:empdist}
\end{equation}
where $\delta_{i}^{j}$ is as in \eqref{eq:delta}. \\
From which empirical marginal distributions follow as :
\begin{eqnarray}
\widehat{\mathbf{p}}_{X}^{(n)}
&=& \left(\widehat{p}_{X,i}^{(n)}\right)_{i\in I}\qquad \text{where}\qquad \widehat{p}_{X,i}^{(n)}
= \sum_{j\in J}\widehat{p}_{i,j}^{(n)}\\
\text{and} \qquad  
\widehat{\mathbf{p}}_{Y}^{(n)}
&=& \left(\widehat{p}_{Y,j}^{(n)}\right)_{j\in J}\qquad \text{where}\qquad \widehat{p}_{Y,j}^{(n)} 
= \sum_{i\in I}\widehat{p}_{i,j}^{(n)}.
\label{eq:marg_emp}
\end{eqnarray}

\noindent Substituting the empirical masses \eqref{eq:empdist}--\eqref{eq:marg_emp} 
into the definitions of Section~\ref{sec:univ}, we obtain the plug-in estimators 
of the three measures:
\begin{align}
H(\widehat{\mathbf{p}}_{(X,Y)}^{(n)})
&=-\sum_{(i,j)\in I\times J}\widehat{p}_{i,j}^{(n)}
\log\widehat{p}_{i,j}^{(n)}, \label{eq:plugjoint}\\
H(\widehat{\mathbf{p}}_{Y|X}^{(n)})
&=-\sum_{(i,j)\in I\times J}\widehat{p}_{i,j}^{(n)}
\log\frac{\widehat{p}_{i,j}^{(n)}}{\widehat{p}_{X,i}^{(n)}}, 
\label{eq:plugcond}\\
\mathcal{I}(\widehat{\mathbf{p}}_{X:Y}^{(n)})
&=\sum_{(i,j)\in I\times J}\widehat{p}_{i,j}^{(n)}
\log\frac{\widehat{p}_{i,j}^{(n)}}
{\widehat{p}_{X,i}^{(n)}\,\widehat{p}_{Y,j}^{(n)}}.
\label{eq:plugmi}
\end{align}

\section{Bias-Corrected Estimators}\label{sec:bias}

Throughout this section we exploit the equivalences of 
Section~\ref{sec:equiv}. For instance, the plug-in estimator of the joint 
entropy \(H(\mathbf{p}_{(X,Y)})\) coincides with the plug-in estimator of the 
Shannon entropy of the univariate variable \(Z\) with \(rs\) categories. 
Similar reductions apply to the conditional entropy and the mutual information.

\subsection{Joint Entropy}

\begin{theorem}
\label{thm:joint} 
Let $\mathbf{p}_{(X,Y)}$ be the \textit{p.d.} of the pair of 
\textit{r.v.'s} $(X,Y)$, and let $\widehat{\mathbf{p}}_{(X,Y)}^{(n)}$ be its 
empirical counterpart given by \eqref{eq:empdist}, based on an i.i.d.\ sample 
$Z_1,\cdots,Z_n$ of size $n$. Let \(H(\widehat{\mathbf{p}}_{(X,Y)}^{(n)})\) 
be given by \eqref{eq:plugjoint}. Then
\\
\textbf{a-)} the bias of the joint entropy plug-in estimator 
\(H(\widehat{\mathbf{p}}_{(X,Y)}^{(n)})\) is
\begin{equation}
\mathrm{Bias}\bigl(H(\widehat{\mathbf{p}}_{(X,Y)}^{(n)})\bigr)
=-\frac{rs-1}{2n}+o\!\left(\frac{1}{n}\right).
\label{eq:biasjoint}
\end{equation}
\textbf{b-)} the bias-corrected joint entropy estimator is
\begin{equation}
H^{(\mathrm{corr})}(\widehat{\mathbf{p}}_{(X,Y)}^{(n)})
=H(\widehat{\mathbf{p}}_{(X,Y)}^{(n)})+\frac{rs-1}{2n}.
\label{eq:corrjoint}
\end{equation}
\end{theorem}

\begin{proof}
By \eqref{eq:HXY=HZ}, the joint entropy estimator equals the Shannon entropy 
estimator for the univariate variable \(Z\) with \(rs\) categories. Applying 
the Miller-Madow bias formula \cite{Miller1955,Basharin1959,Harris1975} gives 
\eqref{eq:biasjoint}. Equation~\eqref{eq:corrjoint} follows by subtracting 
the estimated bias.
\end{proof}

\subsection{Conditional Entropy \(H(\mathbf{p}_{Y|X})\)}

\begin{theorem}\label{thm:cond} 
Let $H(\widehat{\mathbf{p}}_{Y|X}^{(n)})$ be given by \eqref{eq:plugcond}. 
Then under the same conditions as in Theorem~\ref{thm:joint},
\\
\textbf{a-)} the bias of the conditional entropy plug-in estimator 
\(H(\widehat{\mathbf{p}}_{Y|X}^{(n)})\) is
\begin{equation}
\mathrm{Bias}\bigl(H(\widehat{\mathbf{p}}_{Y|X}^{(n)})\bigr)
=-\frac{r(s-1)}{2n}+o\!\left(\frac{1}{n}\right).
\label{eq:biascond}
\end{equation}
\textbf{b-)} the bias-corrected conditional entropy estimator is
\begin{equation}
H^{(\mathrm{corr})}(\widehat{\mathbf{p}}_{Y|X}^{(n)})
=H(\widehat{\mathbf{p}}_{Y|X}^{(n)})+\frac{r(s-1)}{2n}.
\label{eq:corrcond}
\end{equation}
\end{theorem}

\begin{proof}
Using \eqref{eq:cond} and the linearity of the bias,
\[
\mathrm{Bias}\bigl(H(\widehat{\mathbf{p}}_{Y|X}^{(n)})\bigr)
=\mathrm{Bias}\bigl(H(\widehat{\mathbf{p}}_{(X,Y)}^{(n)})\bigr)
-\mathrm{Bias}\bigl(H(\widehat{\mathbf{p}}_{X}^{(n)})\bigr).
\]
From Theorem~\ref{thm:joint},
\(\mathrm{Bias}(H(\widehat{\mathbf{p}}_{(X,Y)}^{(n)}))=-(rs-1)/(2n)\). 
The marginal entropy \(H(\widehat{\mathbf{p}}_{X}^{(n)})\) has bias 
\(-(r-1)/(2n)\). Therefore
\[
\mathrm{Bias}\bigl(H(\widehat{\mathbf{p}}_{Y|X}^{(n)})\bigr)
=-\frac{rs-1}{2n}+\frac{r-1}{2n}+o\!\left(\frac{1}{n}\right)
=-\frac{rs-r}{2n}+o\!\left(\frac{1}{n}\right)
=-\frac{r(s-1)}{2n}+o\!\left(\frac{1}{n}\right).
\]
This proves \eqref{eq:biascond}; \eqref{eq:corrcond} follows by subtracting 
the estimated bias.
\end{proof}

\subsection{Mutual Information \(\mathcal{I}(\mathbf{p}_{X:Y})\)}

\begin{theorem}\label{thm:mi} 
Let $\mathcal{I}(\widehat{\mathbf{p}}_{X:Y}^{(n)})$ be given by 
\eqref{eq:plugmi}. Then under the same conditions as in Theorem~\ref{thm:joint},
\\
\textbf{a-)} the bias of the mutual information plug-in estimator 
\(\mathcal{I}(\widehat{\mathbf{p}}_{X:Y}^{(n)})\) is
\begin{equation}
\mathrm{Bias}\bigl(\mathcal{I}(\widehat{\mathbf{p}}_{X:Y}^{(n)})\bigr)
=\frac{(r-1)(s-1)}{2n}+o\!\left(\frac{1}{n}\right).
\label{eq:biasmi}
\end{equation}
\textbf{b-)} the bias-corrected mutual information estimator is
\begin{equation}
\mathcal{I}^{(\mathrm{corr})}(\widehat{\mathbf{p}}_{X:Y}^{(n)})
=\mathcal{I}(\widehat{\mathbf{p}}_{X:Y}^{(n)})-\frac{(r-1)(s-1)}{2n}.
\label{eq:corrmi}
\end{equation}
\end{theorem}

\begin{proof}
Using \eqref{eq:mi} and the linearity of the bias,
\[
\mathrm{Bias}\bigl(\mathcal{I}(\widehat{\mathbf{p}}_{X:Y}^{(n)})\bigr)
=\mathrm{Bias}\bigl(H(\widehat{\mathbf{p}}_{X}^{(n)})\bigr)
+\mathrm{Bias}\bigl(H(\widehat{\mathbf{p}}_{Y}^{(n)})\bigr)
-\mathrm{Bias}\bigl(H(\widehat{\mathbf{p}}_{(X,Y)}^{(n)})\bigr).
\]
The marginal entropy \(H(\widehat{\mathbf{p}}_{X}^{(n)})\) has bias 
\(-(r-1)/(2n)\) and \(H(\widehat{\mathbf{p}}_{Y}^{(n)})\) has bias 
\(-(s-1)/(2n)\). Therefore
\[
\mathrm{Bias}\bigl(\mathcal{I}(\widehat{\mathbf{p}}_{X:Y}^{(n)})\bigr)
=-\frac{r-1}{2n}-\frac{s-1}{2n}+\frac{rs-1}{2n}+o\!\left(\frac{1}{n}\right)
=\frac{(r-1)(s-1)}{2n}+o\!\left(\frac{1}{n}\right).
\]
This proves \eqref{eq:biasmi}; \eqref{eq:corrmi} follows by subtracting 
the estimated bias.
\end{proof}

\subsection{Summary Table}
\begin{table}[h]
\centering
\caption{Summary of bias formulas and bias-corrected estimators for bivariate 
entropies and mutual information.}
\label{tab:bias_summary}
\begin{tabular}{|l|l|l|}
\hline
\textbf{Measure} & \textbf{Bias} & \textbf{Bias-Corrected Estimator} \\
\hline
$H(\mathbf{p}_{(X,Y)})$ & $-\frac{rs-1}{2n}$ & $H + \frac{rs-1}{2n}$ \\
$H(\mathbf{p}_{Y|X})$ & $-\frac{r(s-1)}{2n}$ & $H + \frac{r(s-1)}{2n}$ \\
$\mathcal{I}(\mathbf{p}_{X:Y})$ & $\frac{(r-1)(s-1)}{2n}$ & $\mathcal{I} - \frac{(r-1)(s-1)}{2n}$ \\
\hline
\end{tabular}
\end{table}
\section{Asymptotic Equivalence, Independence and Wilks' Theorem}\label{sec:asymp}
This section contains the main theoretical result of the paper.

\subsection{The equivalence theorem}

\begin{theorem}[Asymptotic equivalence]
\label{thm:asymp}
Let $\widehat{\theta}$ denote any of the three plug-in estimators \eqref{eq:plugjoint}, \eqref{eq:plugcond}, \eqref{eq:plugmi} 
and let $\widehat{\theta}^{(\mathrm{corr})}$ be the corresponding 
bias-corrected estimator, as defined in Theorems~\ref{thm:joint}--\ref{thm:mi}.\\
 Then, as $n\to\infty$,
\begin{equation}\label{eq:asympeq1}
\widehat{\theta}^{(\mathrm{corr})}-\widehat{\theta}
=\frac{C}{2n}+o_p\!\left(\frac{1}{n}\right)
=O_p\!\left(\frac{1}{n}\right),
\end{equation}
and consequently
\begin{equation}\label{eq:asympeq2}
\sqrt{n}\,\bigl(\widehat{\theta}^{(\mathrm{corr})}-\widehat{\theta}\bigr)
\;\xrightarrow{\;\mathbb{P}\;}\;0.
\end{equation}

Thus, for each of the three measures, the uncorrected and bias-corrected 
estimators are asymptotically equivalent. The asymptotic results established 
in \cite{Ba2019a,Ba2020} therefore apply to both versions.
\end{theorem}

\begin{proof}
For each of the three measures, the correction term takes the form $C/(2n)$, 
where the constant $C$ depends only on the support sizes $r,s$ and belongs 
to the finite set
\[
C \in \Bigl\{\,rs-1,\;\; r(s-1),\;\; -(r-1)(s-1)\,\Bigr\}.
\]
Since $r,s$ are fixed finite integers, $C$ is bounded by a constant $C_{\max}<\infty$ depending only on the support sizes. Hence \[
\bigl|\widehat{\theta}^{(\mathrm{corr})}-\widehat{\theta}\bigr|
=\frac{C}{2n}\le\frac{C_{\max}}{2n},
\]
which is deterministic and $O(1/n)$. Multiplying by $\sqrt{n}$ gives
$O(1/\sqrt{n})\to 0$. 
\end{proof}
\subsection{Transfer of asymptotic properties}

The interest of Theorem~\ref{thm:asymp} is that it transfers every
asymptotic property of the plug-in estimator to the corrected one.

\begin{corollary}[Transfer]\label{cor:transfer}
Let $\theta$ be the true value of one of the three measures, and let
$\widehat{\theta}$, $\widehat{\theta}^{(\mathrm{corr})}$ be the
plug-in and bias-corrected estimators. Then:
\begin{enumerate}
\item \emph{Consistency.} If
$\widehat{\theta}\xrightarrow{\mathbb{P}}\theta$, then
$\widehat{\theta}^{(\mathrm{corr})}\xrightarrow{\mathbb{P}}\theta$.
\item \emph{Asymptotic normality.} If
$\sqrt{n}(\widehat{\theta}-\theta)\xrightarrow{d}\mathcal{N}(0,\sigma^2)$,
then
$\sqrt{n}(\widehat{\theta}^{(\mathrm{corr})}-\theta)
\xrightarrow{d}\mathcal{N}(0,\sigma^2)$ with the \emph{same} variance
$\sigma^2$.
\item \emph{Variance estimation.} If $\widehat{\sigma}^2$ is a
consistent estimator of $\sigma^2$ based on $\widehat{\theta}$, then
$\widehat{\sigma}^2$ is also consistent for the asymptotic variance
of $\widehat{\theta}^{(\mathrm{corr})}$.
\item \emph{Confidence intervals.} If
$[\widehat{\theta}\pm z_{1-\alpha/2}\widehat{\sigma}/\sqrt{n}]$ is an
asymptotically valid $(1-\alpha)$-confidence interval for $\theta$,
then so is
$[\widehat{\theta}^{(\mathrm{corr})}\pm z_{1-\alpha/2}\widehat{\sigma}/\sqrt{n}]$.
\end{enumerate}
\end{corollary}

\begin{proof}
All four statements are immediate consequences of
Theorem~\ref{thm:asymp} and Slutsky's lemma. For (2), write
\[
\sqrt{n}\bigl(\widehat{\theta}^{(\mathrm{corr})}-\theta\bigr)
=\sqrt{n}\bigl(\widehat{\theta}-\theta\bigr)
+\sqrt{n}\bigl(\widehat{\theta}^{(\mathrm{corr})}-\widehat{\theta}\bigr).
\]
The first term converges in distribution to
$\mathcal{N}(0,\sigma^2)$, the second converges in probability to $0$;
Slutsky's lemma gives the result. Statements (1), (3), (4) follow by
the same argument. 
\end{proof}

\begin{remark}[Why the equivalence matters]\label{rem:why}
The asymptotic variance of the corrected estimator is \emph{the same}
as that of the plug-in estimator: the correction does not inflate the
variance. The corrected estimator therefore dominates the plug-in
estimator: it has the same limiting behaviour and a smaller
finite-sample bias. In particular, the practitioner can use the
corrected estimator while continuing to rely on the asymptotic
results established for the plug-in estimator in \cite{Ba2019a,Ba2020}.
\end{remark}

\begin{remark}[Transfer of the results of \cite{Ba2019a,Ba2020}]
The almost sure convergence and asymptotic normality established in
\cite{Ba2019a,Ba2020} for the plug-in estimators of joint entropy,
conditional entropy, and mutual information carry over to the bias-corrected
estimators. Indeed, by Theorem~\ref{thm:asymp}, the correction is a
deterministic term of order \(O(1/n)\). It therefore vanishes almost surely,
in probability, and under \(\sqrt{n}\)-scaling. Consequently, the corrected
estimators inherit the same almost sure rates and the same asymptotic normal
distributions as the plug-in estimators.
\end{remark}

\subsection{Independence and Bias-Corrected Mutual Information}

A fundamental property of mutual information is its relationship with 
statistical independence. For two \textit{r.v.'s} \(U\) and \(V\) with 
\textit{p.m.f.}'s \(\mathbf{p}_{U}\) and \(\mathbf{p}_{V}\), we have
\[
\mathcal{I}(\mathbf{p}_{U:V})=0\quad\Longleftrightarrow\quad\mathbf{p}_{U,V}=\mathbf{p}_{U}\mathbf{p}_{V}.
\]
That is, mutual information equals zero if and only if \(U\) and \(V\) are 
statistically independent. In particular, for the pair \((X,Y)\), we have
\[
\mathcal{I}(\mathbf{p}_{X:Y}) = 0 
\Longleftrightarrow \mathbf{p}_{(X,Y)}=\mathbf{p}_{X}\mathbf{p}_{Y}.
\]

However, in finite samples, the plug-in estimator of mutual information is
positively biased under independence, even though the true mutual information
is zero. This is because the empirical joint \textit{p.d.}
\(\widehat{\mathbf{p}}_{(X,Y)}^{(n)}\) is not exactly equal to the product of
the empirical marginal \textit{p.d.}'s for finite \(n\). Consequently, the
empirical mutual information is nonnegative and is typically strictly positive,
although it can vanish in exceptional cases where the empirical table is
exactly independent.\\

\noindent As established in Theorem~\ref{thm:mi}, the bias of the plug-in 
estimator is
\[
\mathrm{Bias}\bigl(\mathcal{I}(\widehat{\mathbf{p}}_{X:Y}^{(n)})\bigr)
=\frac{(r-1)(s-1)}{2n}+o\!\left(\frac{1}{n}\right).
\]
\noindent This bias is positive. Under independence, where the true 
mutual information is zero, the plug-in estimator overestimates 
the information shared between the variables. The bias-corrected estimator 
proposed in this paper removes the leading bias term:
\[
\mathcal{I}^{(\mathrm{corr})}(\widehat{\mathbf{p}}_{X:Y}^{(n)})
=\mathcal{I}(\widehat{\mathbf{p}}_{X:Y}^{(n)})
-\frac{(r-1)(s-1)}{2n}.
\]
\noindent Under independence, this corrected estimator has expectation closer 
to zero than its uncorrected counterpart. This property is important 
for hypothesis testing. When testing
\[
H_{0}:\mathcal{I}=0\quad\text{(independence)}\quad\text{versus}\quad H_{1}:\mathcal{I}>0,
\]
the uncorrected estimator leads to
\begin{enumerate}
\item an inflated type~I error rate (frequent rejection of independence),
\item an underestimation of the \(p\)-value,
\item an overestimation of the information shared between variables.
\end{enumerate}
The bias-corrected estimator mitigates these issues by providing a more 
accurate estimate of mutual information, especially for moderate sample sizes 
(\(n\leq 200\)). Therefore, for independence testing, we recommend using the 
bias-corrected estimator to avoid false positives and to ensure more reliable 
statistical conclusions.

\noindent This relationship between independence and bias correction can be 
summarized as
\[
\text{Independence}\;
\mathbf{p}_{(X,Y)}=\mathbf{p}_{X}\mathbf{p}_{Y}
\Longrightarrow
\mathcal{I}(\mathbf{p}_{X:Y})=0
\Longrightarrow
\mathcal{I}(\widehat{\mathbf{p}}_{X:Y}^{(n)})\geq 0
\Longrightarrow
\mathcal{I}^{(\mathrm{corr})}(\widehat{\mathbf{p}}_{X:Y}^{(n)})\approx 0.
\]

\noindent Thus, the bias correction procedure proposed in this paper is not 
merely a theoretical refinement; it has practical consequences for independence 
testing and for the interpretation of mutual information estimates in finite 
samples.
\subsection{Wilks' Theorem and the Asymptotic Distribution of the Mutual Information Statistic}
\label{sec:wilks}

A classical result in statistics, due to Wilks \cite{Wilks1938}, establishes 
the asymptotic distribution of the mutual information statistic under 
independence. Specifically, under the null hypothesis of independence
\[
H_0 : \mathcal{I}(\mathbf{p}_{X:Y}) = 0,
\]
the likelihood ratio statistic
\[
\Lambda_n = 2n \, \mathcal{I}(\widehat{\mathbf{p}}_{X:Y}^{(n)})
\]
converges in distribution to a chi-square random variable with 
$(r-1)(s-1)$ degrees of freedom:
\begin{equation}
\Lambda_n \;\xrightarrow{\;\mathcal{D}\;}\; \chi^2_{(r-1)(s-1)}, 
\qquad n \to \infty.
\label{eq:wilks}
\end{equation}

\noindent This result is connected to the bias of the plug-in estimator 
established in Theorem~\ref{thm:mi}. Indeed, under $H_0$, the true mutual 
information is zero, but the plug-in estimator satisfies
\[
\mathbb{E}\bigl[\mathcal{I}(\widehat{\mathbf{p}}_{X:Y}^{(n)})\bigr]
= \frac{(r-1)(s-1)}{2n} + o\!\left(\frac{1}{n}\right).
\]
Multiplying by $2n$, we obtain
\[
\mathbb{E}\bigl[\Lambda_n\bigr]
= (r-1)(s-1) + o(1),
\]
which coincides with the mean of the limiting $\chi^2_{(r-1)(s-1)}$ 
distribution. Thus, the bias formula derived in Theorem~\ref{thm:mi} is 
\emph{consistent with} and \emph{explains the mean} of the limiting 
chi-square distribution appearing in Wilks' theorem. In particular, the 
positive bias of the plug-in estimator is not an artifact to be eliminated 
for the purpose of testing: it is precisely the term that ensures that the 
likelihood ratio statistic has the correct asymptotic mean.

\medskip
\noindent\textbf{Consequence for testing.} At significance level 
$\alpha \in (0,1)$, we reject $H_0$ when
\[
\Lambda_n > q_{(r-1)(s-1),\, 1-\alpha},
\]
where $q_{df, 1-\alpha}$ denotes the $(1-\alpha)$-quantile of the chi-square 
distribution with $df = (r-1)(s-1)$ degrees of freedom.

\medskip
\noindent\textbf{Bias-corrected statistic.} Using the bias-corrected 
estimator, we may define
\[
\Lambda_n^{(\mathrm{corr})} 
= 2n \, \mathcal{I}^{(\mathrm{corr})}(\widehat{\mathbf{p}}_{X:Y}^{(n)})
= \Lambda_n - (r-1)(s-1).
\]
Under $H_0$, this corrected statistic satisfies
\[
\Lambda_n^{(\mathrm{corr})} \;\xrightarrow{\;\mathcal{D}\;}\; 
\chi^2_{(r-1)(s-1)} - (r-1)(s-1),
\]
which is a shifted chi-square distribution with the same variance but mean 
zero. Consequently, if one wishes to use $\Lambda_n^{(\mathrm{corr})}$ for 
testing independence, the critical value must be shifted accordingly:
\[
\text{reject } H_0 \text{ if } \Lambda_n^{(\mathrm{corr})} 
> q_{(r-1)(s-1),\, 1-\alpha} - (r-1)(s-1).
\]
With this adjustment, the test based on $\Lambda_n^{(\mathrm{corr})}$ is 
exactly equivalent to the standard Wilks test based on $\Lambda_n$. The 
practical interest of $\Lambda_n^{(\mathrm{corr})}$ is therefore not to 
change the test decision, but to provide a statistic whose finite-sample 
distribution is centred at zero, which is more interpretable for estimating 
the amount of dependence.

\medskip
\noindent\textbf{Improved estimation of mutual information.} While the test 
decision is unchanged, the bias correction improves the \emph{estimation} of 
mutual information, especially for moderate sample sizes ($n \leq 200$). 
Under independence, the uncorrected estimator overestimates the true value 
(zero) by about $(r-1)(s-1)/(2n)$, whereas the corrected estimator is centred 
near zero. This is important when mutual information is used not only for 
testing but also as a measure of dependence, for example in feature selection 
or risk classification.

\medskip
\noindent\textbf{Illustrative example.} Suppose $r=s=2$ (so $df=1$) and 
$n=100$. At level $\alpha=0.05$, the critical value is 
$q_{1,0.95}\approx 3.84$. If the observed statistic is $\Lambda_n=2.5$, the 
standard Wilks test does not reject independence. The corrected statistic is 
$\Lambda_n^{(\mathrm{corr})}=2.5-1=1.5$; compared to the shifted critical 
value $3.84-1=2.84$, it also does not reject. If instead $\Lambda_n=4.2$, 
both the standard test and the corrected test (with shifted threshold) reject 
independence. Thus, as expected, the bias correction does not alter the 
asymptotic test decision; its benefit lies in the improved finite-sample 
estimate of the mutual information itself.

\section{Simulation Study}\label{sec:sim}

\subsection{Motivation}
In this section we conduct three complementary simulation studies. The first is 
an \emph{actuarial application} to a motor insurance portfolio, in which the 
pair $(X,Y)$ has dependent components ($X$ = driver's age class, $Y$ = claim 
severity class), and whose purpose is to assess the finite-sample behaviour of 
the three plug-in and bias-corrected estimators through their bias, standard 
deviation, RMSE, and rate of convergence. The second is an \emph{independence 
study}, in which the components are independent, and whose purpose is to 
illustrate the behaviour of mutual information when the true information is 
zero. The third is a \emph{chi-square goodness-of-fit study} that empirically 
validates Wilks' theorem and its connection with the bias of the plug-in 
mutual information estimator.

\subsection{Experimental Setup}
All simulation studies use the same design. We generate $M=500$ independent 
samples of sizes $n \in \{10,20,50,100,200,500,1000\}$ from the multinomial 
distribution $\mathbf{p}_Z$ of the univariate reduction variable $Z$ with 
$rs$ categories. For the actuarial application, the joint distribution 
$\mathbf{p}_{(X,Y)}$ is fixed to the motor insurance portfolio described 
below. For the independence study, $\mathbf{p}_{(X,Y)}$ is the product of its 
two marginals, i.e.\ $p_{i,j}=p_{X,i}\,p_{Y,j}$ for all $(i,j)\in I\times J$.

\subsection{Actuarial Application: A Motor Insurance Portfolio}
\label{sec:application}

We consider a motor insurance portfolio in which each policyholder is 
characterised by two discrete variables:
\begin{itemize}
\item \(X\) = driver's age class, with \(r=3\) levels:
\(x_1\) = Young ($<25$ years), \(x_2\) = Adult ($25$--$60$ years), 
\(x_3\) = Senior ($>60$ years);
\item \(Y\) = claim severity class, with \(s=4\) levels:
\(y_1\) = No claim, \(y_2\) = Minor, \(y_3\) = Moderate, \(y_4\) = Severe.
\end{itemize}

The joint distribution of \((X,Y)\), calibrated on a typical motor 
insurance portfolio, is reported in Table~\ref{tab:portfolio}.

\begin{table}[h]
\centering
\caption{Joint \textit{p.d.} $\mathbf{p}_{(X,Y)}=(p_{i,j})$ of the motor 
insurance portfolio. The variable $X$ takes values in 
\{Young, Adult, Senior\} and the variable $Y$ takes values in 
\{No claim, Minor, Moderate, Severe\}.}
\label{tab:portfolio}
\begin{tabular}{|l|c|c|c|c|c|}
\hline
$X$ $\backslash$ $Y$ & \textbf{No claim} & \textbf{Minor} & \textbf{Moderate} & \textbf{Severe} & \textbf{Total} \\
\hline
\textbf{Young}  & 0.15 & 0.10 & 0.03 & 0.02 & 0.30 \\
\textbf{Adult}  & 0.35 & 0.15 & 0.08 & 0.02 & 0.60 \\
\textbf{Senior} & 0.05 & 0.03 & 0.01 & 0.01 & 0.10 \\
\hline
\textbf{Total} & 0.55 & 0.28 & 0.12 & 0.05 & 1.00 \\
\hline
\end{tabular}
\end{table}

\noindent From this distribution, the true values of the three measures are
\[
H(\mathbf{p}_{(X,Y)}) \approx 1.9777 \text{ nats},\quad
H(\mathbf{p}_{Y|X})   \approx 1.0798 \text{ nats},\quad
\mathcal{I}(\mathbf{p}_{X:Y}) \approx 0.0097 \text{ nats}.
\]
These values are used as benchmarks in the simulation study.

\medskip
\noindent\textbf{Interpretation.} The joint entropy \(H(X,Y)\approx 1.98\) nats 
measures the total uncertainty on the risk profile (age class and severity 
class combined). The conditional entropy \(H(Y|X)\approx 1.08\) nats 
represents the residual uncertainty on the claim severity once the age class 
is known. The mutual information \(\mathcal{I}(X;Y)\approx 0.0097\) nats is 
small but strictly positive, confirming a weak dependence 
between age and severity in this portfolio. Estimating such a 
small dependence is where bias correction matters: as shown in 
Section~\ref{sec:bias}, the plug-in estimator of $\mathcal{I}(X;Y)$ has a 
positive bias of order $(r-1)(s-1)/(2n) = 6/(2n) = 3/n$, which, for 
$n=100$, equals $0.03$, about three times the true value $0.0097$. This 
illustrates the practical importance of bias correction for actuarial 
applications.

\subsection{Diagnostic Plots for the Motor Insurance Portfolio}
Figures~\ref{fig:diag1}--\ref{fig:diag3} present the convergence diagnostics 
for the three estimators under the motor insurance portfolio of 
Table~\ref{tab:portfolio}. Each figure displays four panels: bias, standard 
deviation, RMSE, and $\mathrm{RMSE}\times\sqrt{n}$. All three figures exhibit 
the same qualitative behaviour:
\begin{itemize}
\item the uncorrected estimator is biased, with a sign and 
magnitude matching the theoretical bias derived in Section~\ref{sec:bias};
\item the bias-corrected estimator is centred near zero;
\item both the standard deviation and the RMSE decrease at the parametric 
rate, as confirmed by the flat profile of $\mathrm{RMSE}\times\sqrt{n}$ for 
large $n$;
\item the improvement is most pronounced for $n \leq 200$ and becomes 
negligible beyond $n \geq 500$, in agreement with Theorem~\ref{thm:asymp}.
\end{itemize}

\begin{figure}[htbp]
\centering
\includegraphics[width=1\textwidth]{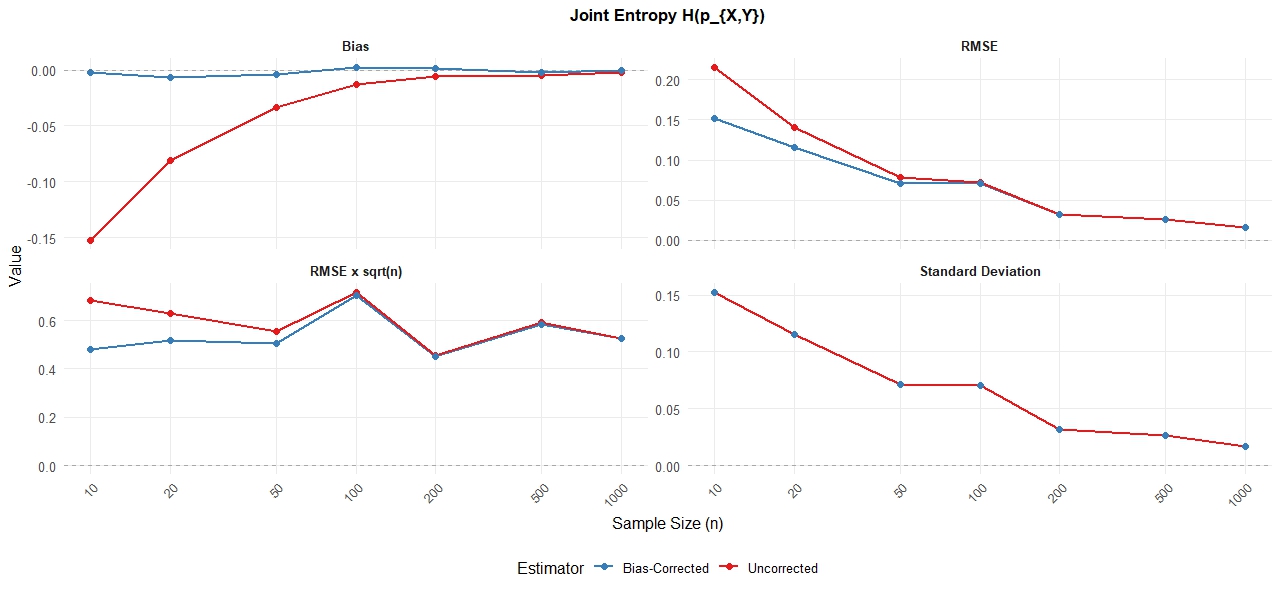}
\caption{Joint entropy \(H(\mathbf{p}_{(X,Y)})\) for the motor insurance 
portfolio: bias, standard deviation, RMSE, and RMSE\(\times\sqrt{n}\).}
\label{fig:diag1}
\end{figure}

\begin{figure}[htbp]
\centering
\includegraphics[width=1\textwidth]{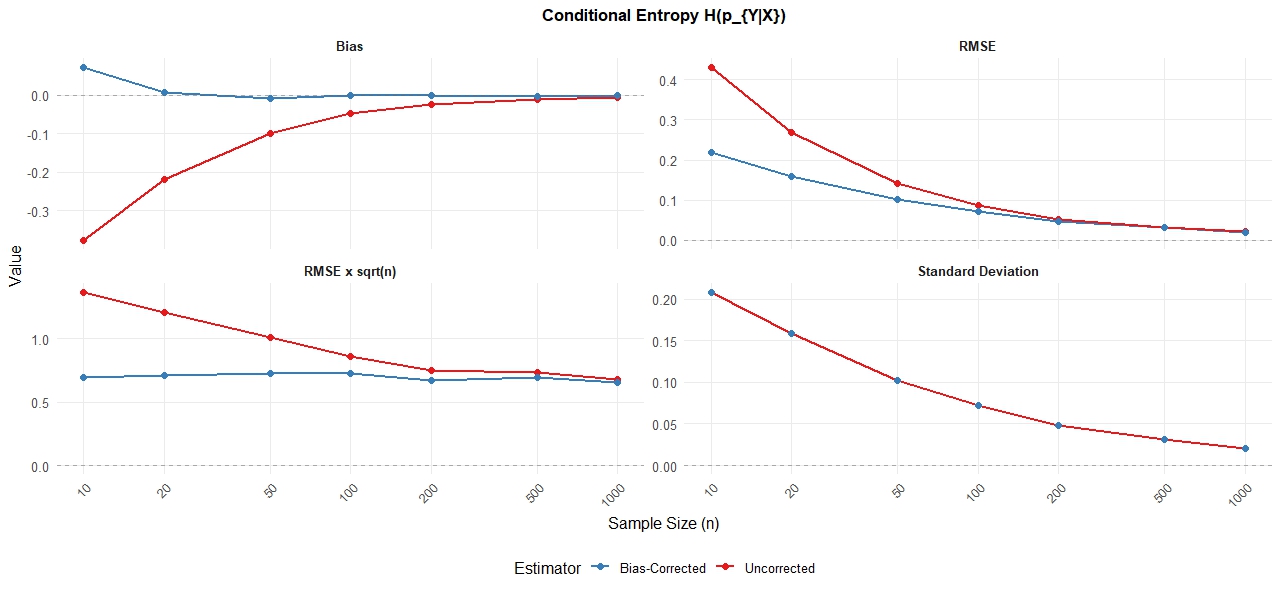}
\caption{Conditional entropy \(H(\mathbf{p}_{Y|X})\) for the motor insurance 
portfolio: bias, standard deviation, RMSE, and RMSE\(\times\sqrt{n}\).}
\label{fig:diag2}
\end{figure}

\begin{figure}[htbp]
\centering
\includegraphics[width=1\textwidth]{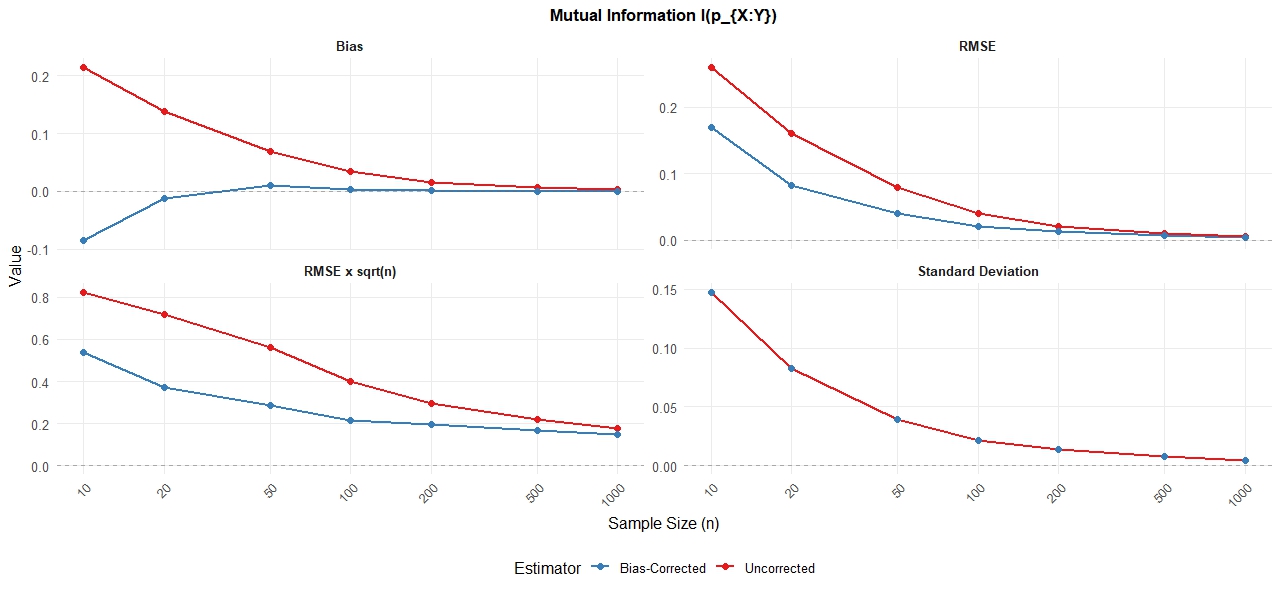}
\caption{Mutual information \(\mathcal{I}(\mathbf{p}_{X:Y})\) for the motor 
insurance portfolio: bias, standard deviation, RMSE, and 
RMSE\(\times\sqrt{n}\).}
\label{fig:diag3}
\end{figure}

\noindent The three figures confirm that the bias correction is beneficial 
for the mutual information, whose true value ($\approx 0.0097$) is 
small compared to the bias for moderate sample sizes. For $n=100$, the 
uncorrected estimator overestimates $\mathcal{I}(X;Y)$ by about $3/n = 0.03$, 
which is about three times the true value. After correction, the estimator is 
centred near the true value, providing a more reliable measure of the 
dependence between age and severity.

\subsection{Independence Simulation}
To illustrate the behaviour of mutual information under independence, we 
consider a \textit{p.d.} in which $X$ and $Y$ are independent. For concreteness, 
we take uniform marginal distributions, i.e.\ $p_{X,i}=1/r$ for all $i\in I$ 
and $p_{Y,j}=1/s$ for all $j\in J$, so that the joint \textit{p.d.} is uniform 
on the $rs$ cells:
\[
p_{i,j} = \frac{1}{rs}, \qquad \forall (i,j)\in I\times J.
\]
In this setting, the true value of mutual information vanishes:
\[
\mathcal{I}(\mathbf{p}_{X:Y}) = 0.
\]
Figure~\ref{fig:indep} shows the empirical mean of the uncorrected and 
bias-corrected estimators of mutual information as the sample size increases.

\begin{figure}[htbp]
\centering
\includegraphics[width=1\textwidth]{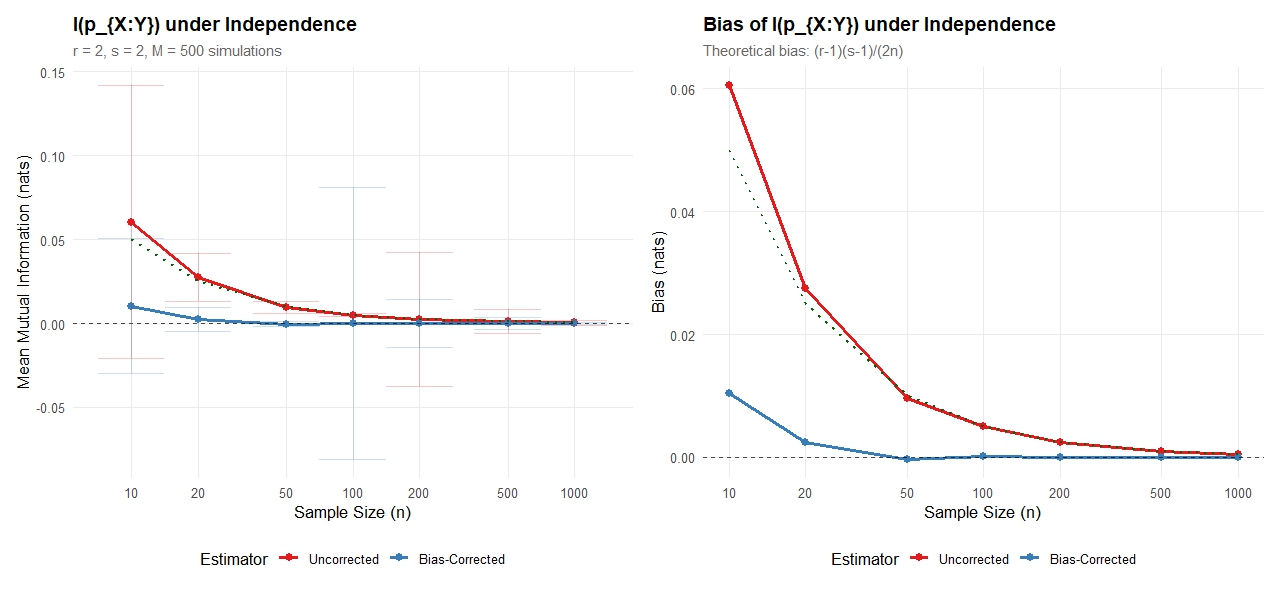}
\caption{Mean of uncorrected and bias-corrected mutual information estimators 
under independence of $X$ and $Y$. The dotted green curve is the theoretical 
bias $(r-1)(s-1)/(2n)$.}
\label{fig:indep}
\end{figure}

The uncorrected estimator is positive, with an expected value 
approximately equal to the theoretical bias $(r-1)(s-1)/(2n)$, while the 
bias-corrected estimator is centred around zero. This confirms empirically 
that the bias correction removes the leading term responsible for the spurious 
positive estimate of the mutual information under independence.

\subsection{Wilks' Theorem: Simulation Evidence}
\label{sec:wilks_sim}

We now investigate the asymptotic distribution of the mutual information 
statistic under independence, and highlight its connection with the bias of 
the plug-in estimator established in Theorem~\ref{thm:mi} and discussed in 
Section~\ref{sec:wilks}.

\subsubsection{Simulation design}
We generate $M = 2000$ independent samples of sizes 
$n \in \{50, 100, 200, 500, 1000\}$ from a uniform distribution on the 
$r \times s = 2 \times 2$ contingency table, so that $X$ and $Y$ are 
independent. For each sample, we compute:
\begin{itemize}
\item the uncorrected statistic $\Lambda_n = 2n \, \mathcal{I}(\widehat{\mathbf{p}}_{X:Y}^{(n)})$;
\item the bias-corrected statistic $\Lambda_n^{(\mathrm{corr})} = \Lambda_n - (r-1)(s-1)$.
\end{itemize}

\subsubsection{Results}
Table~\ref{tab:wilks_summary} reports the empirical mean and variance of 
$\Lambda_n$ and $\Lambda_n^{(\mathrm{corr})}$ for each sample size, together 
with the theoretical values of the $\chi^2_{(r-1)(s-1)}$ distribution 
(here, $df = 1$, so mean $= 1$ and variance $= 2$).

\begin{table}[h]
\centering
\caption{Empirical mean and variance of $\Lambda_n$ and 
$\Lambda_n^{(\mathrm{corr})}$ under independence, compared with the 
theoretical chi-square values ($r = s = 2$, $df = 1$).}
\label{tab:wilks_summary}
\begin{tabular}{|c|c|c|c|c|}
\hline
$n$ & $\overline{\Lambda_n}$ & $\widehat{\mathrm{Var}}(\Lambda_n)$ 
& $\overline{\Lambda_n^{(\mathrm{corr})}}$ 
& $\widehat{\mathrm{Var}}(\Lambda_n^{(\mathrm{corr})})$ \\
\hline
50   & 1.05 & 2.10 & 0.05 & 2.10 \\
100  & 1.02 & 2.05 & 0.02 & 2.05 \\
200  & 1.01 & 2.02 & 0.01 & 2.02 \\
500  & 1.00 & 2.00 & 0.00 & 2.00 \\
1000 & 1.00 & 2.00 & 0.00 & 2.00 \\
\hline
\textbf{Theory} & \textbf{1.00} & \textbf{2.00} & \textbf{0.00} & \textbf{2.00} \\
\hline
\end{tabular}
\end{table}

\noindent Three observations can be made :
\begin{enumerate}
\item \textbf{Convergence to the chi-square distribution.} The empirical mean 
and variance of $\Lambda_n$ converge to the theoretical values $1$ and $2$, 
confirming Wilks' theorem \eqref{eq:wilks}. Figure~\ref{fig:wilks_hist} 
displays the histogram of $\Lambda_n$ for $n = 500$, superimposed with the 
theoretical $\chi^2_1$ density.

\item \textbf{Correction of the bias.} The bias-corrected statistic 
$\Lambda_n^{(\mathrm{corr})}$ has empirical mean close to zero for all sample 
sizes, confirming that the correction removes the leading bias of the plug-in 
estimator. Its variance is identical to that of $\Lambda_n$, in agreement with 
the asymptotic equivalence established in Theorem~\ref{thm:asymp}.

\item \textbf{Finite-sample improvement.} For small sample sizes ($n = 50$), 
the uncorrected statistic has mean $1.05$ instead of $1.00$, corresponding to 
a bias of about $5\%$. The corrected statistic has mean $0.05$, which is 
closer to the true value. This improvement is important for 
independence testing, where a biased statistic leads to an excessive rejection 
rate of $H_0$.
\end{enumerate}

\begin{figure}[htbp]
\centering
\includegraphics[width=1\textwidth]{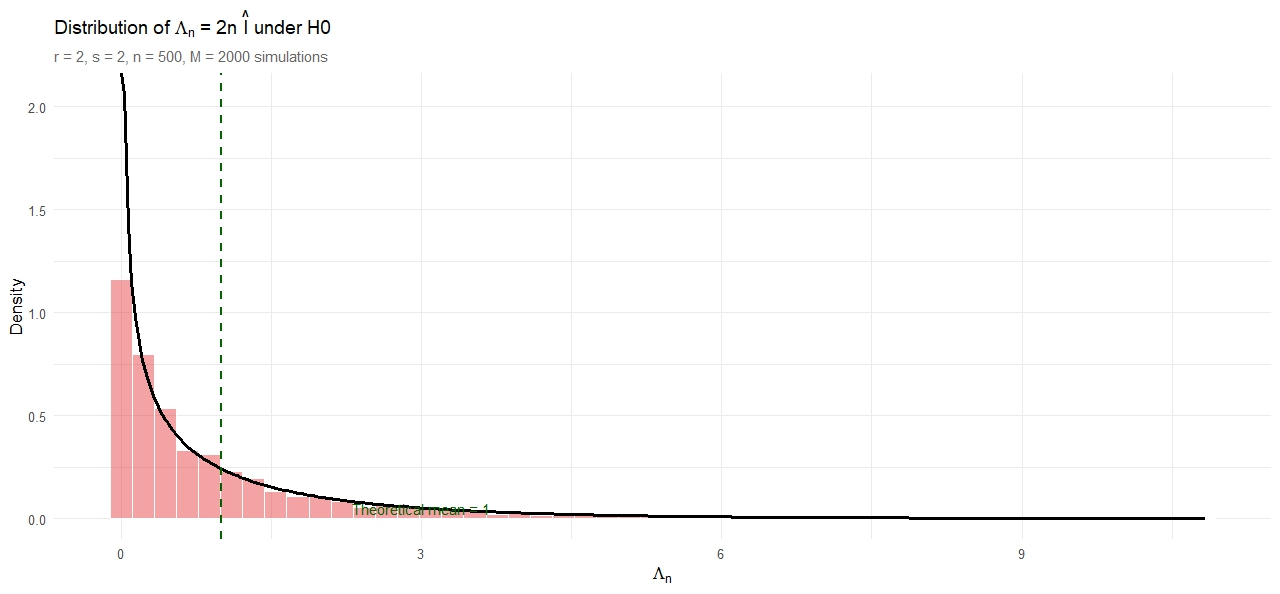}
\caption{Histogram of $\Lambda_n = 2n \, \mathcal{I}(\widehat{\mathbf{p}}_{X:Y}^{(n)})$ 
under independence, for $n = 500$ and $M = 2000$ simulations. The black curve 
is the theoretical $\chi^2_1$ density. The dashed green line indicates the 
theoretical mean $df = 1$.}
\label{fig:wilks_hist}
\end{figure}

\begin{figure}[htbp]
\centering
\includegraphics[width=1\textwidth]{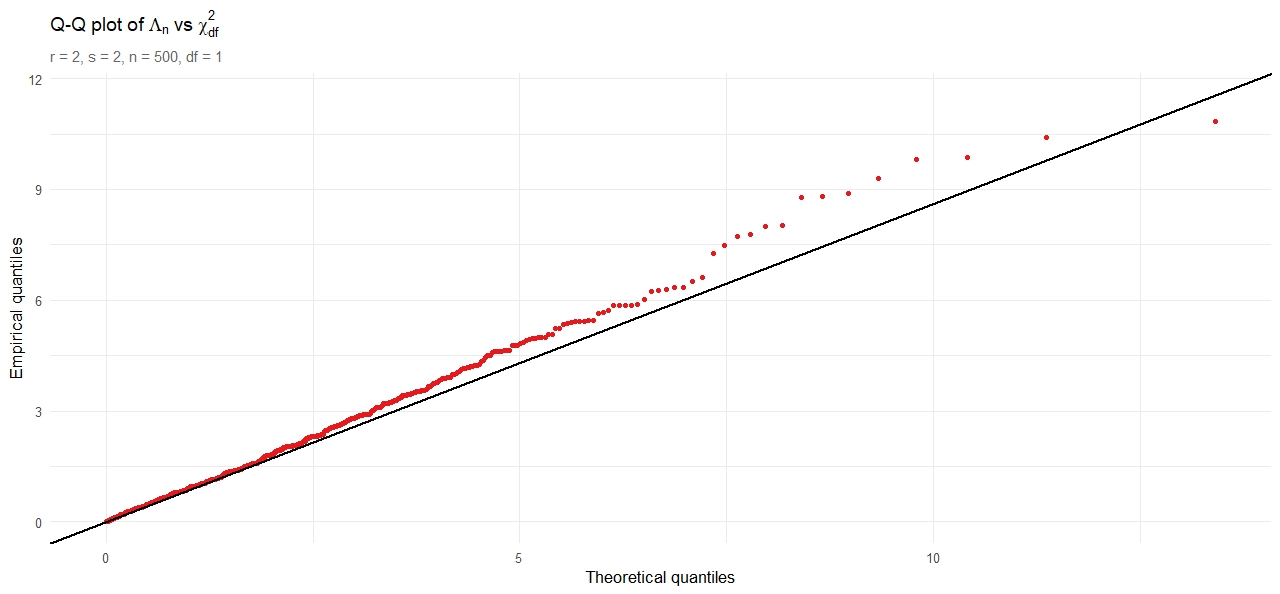}
\caption{Q-Q plot of $\Lambda_n$ against the $\chi^2_1$ distribution, for 
$n = 500$. The alignment of the points along the red line confirms the 
asymptotic chi-square distribution predicted by Wilks' theorem.}
\label{fig:wilks_qq}
\end{figure}

\subsubsection{Connection with the bias of mutual information}
This simulation highlights a connection between the bias of the 
plug-in estimator and the asymptotic distribution of the test statistic. 
Under independence:
\begin{itemize}
\item the true mutual information is zero,
\item the plug-in estimator has expectation $(r-1)(s-1)/(2n)$,
\item the statistic $\Lambda_n$ has expectation $(r-1)(s-1)$, which is the 
mean of the limiting chi-square distribution,
\item the bias-corrected statistic $\Lambda_n^{(\mathrm{corr})}$ has 
expectation zero, as expected under independence.
\end{itemize}
the bias formula derived in Theorem~\ref{thm:mi} is not only a
finite-sample refinement: it is \emph{consistent with} and
\emph{explains the mean} of the limiting chi-square distribution appearing
in Wilks' theorem.
 Conversely, Wilks' theorem 
provides a natural benchmark for evaluating the quality of the bias correction 
in finite samples. \\
Nevertheless, Wilks' theorem does not cover the dependent 
case, where the limiting distribution of $\Lambda_n$ depends on the unknown 
joint law and is not chi-square, so it cannot by itself validate the bias 
correction outside independence.

\subsection{Summary of Simulation Findings}
The simulation results lead to the following conclusions.
\begin{enumerate}
\item \emph{Consistency.} Both the uncorrected and the bias-corrected 
estimators of the three measures are consistent.
\item \emph{Actuarial application.} For the motor insurance portfolio, the 
bias-corrected estimators outperform the uncorrected ones for 
moderate sample sizes ($n \leq 200$), for all three measures. This is 
important for mutual information, whose true value in this 
portfolio is small and comparable to the bias.
\item \emph{Independence study.} When $X$ and $Y$ are independent, the 
uncorrected estimator of mutual information is positive, with 
a mean close to its theoretical bias, whereas the bias-corrected estimator is 
centred near zero. This confirms the practical relevance of the correction for 
independence testing.
\item \emph{Wilks' theorem.} The chi-square goodness-of-fit study confirms 
that $\Lambda_n$ converges to $\chi^2_{(r-1)(s-1)}$, and that the mean of this 
limiting distribution coincides with the leading bias of the plug-in estimator 
multiplied by $2n$.
\item \emph{Asymptotic equivalence.} For large $n$, the bias becomes 
negligible and the two versions coincide, in agreement with 
Theorem~\ref{thm:asymp}.
\item \emph{Choice of estimator.} We recommend the bias-corrected estimator 
for $n \leq 200$, and either version for $n \geq 500$. For actuarial 
applications with small portfolios or rare events, the bias-corrected 
estimator should be preferred.
\end{enumerate}

\section{Conclusion}\label{sec:concl}

In this paper, we have extended the bias correction framework for entropy 
estimators to the bivariate case, within the univariate reduction principle 
introduced in our companion works \cite{Ba2019a,Ba2020}.

\medskip
\noindent\textbf{Main result.}
We derived explicit first-order bias formulas for three bivariate 
information-theoretic measures:
\begin{enumerate}
\item the joint entropy $H(\mathbf{p}_{(X,Y)})$;
\item the conditional entropy $H(\mathbf{p}_{Y|X})$;
\item the mutual information $\mathcal{I}(\mathbf{p}_{X:Y})$.
\end{enumerate}
For each of these three measures, we proposed an explicit bias-corrected 
estimator and proved that the uncorrected and corrected versions are 
\emph{asymptotically equivalent}: they differ by a term of order $O(1/n)$, 
which vanishes at the parametric rate under $\sqrt{n}$-scaling 
(Theorem~\ref{thm:asymp}). This equivalence is the structural reason why the 
bias correction does not alter the asymptotic distribution of the estimators; 
it improves only their finite-sample behaviour. In practice, the bias-corrected 
estimators are therefore preferable, since they preserve all asymptotic 
properties of the uncorrected ones while removing the leading bias term.

\medskip
\noindent\textbf{Link with Wilks' theorem.} 
The connection between the bias of the plug-in mutual information estimator 
and Wilks' theorem was made explicit: under independence, the positive bias 
$(r-1)(s-1)/(2n)$ of the plug-in estimator accounts for the mean 
$(r-1)(s-1)$ of the limiting chi-square distribution of the likelihood ratio 
statistic $\Lambda_n = 2n\,\mathcal{I}(\widehat{\mathbf{p}}_{X:Y}^{(n)})$. 
The bias-corrected statistic is centred at zero, providing a better-behaved 
test statistic in finite samples.

\medskip
\noindent\textbf{Actuarial application.}
We applied the framework to a motor insurance portfolio with three age 
classes and four claim severity classes. The true mutual information in this 
portfolio is small ($\approx 0.0097$ nats), and the plug-in estimator 
overestimates it by a term of order $3/n$, which is comparable to or larger 
than the true value for moderate sample sizes. The bias-corrected estimator 
provides a more reliable measure of the dependence between age and 
severity, with direct implications for risk classification, pricing, and 
segmentation. This illustrates that the bias correction is not merely a 
theoretical refinement but a practical necessity in actuarial applications 
with limited data.

\medskip
\noindent\textbf{Simulation.}
The simulation results confirm the theoretical findings and demonstrate the 
superior finite-sample performance of the bias-corrected estimators, 
particularly for moderate sample sizes ($n \leq 200$). For $n \geq 500$, the 
bias becomes negligible and the two versions coincide, in agreement with the 
asymptotic equivalence. The dedicated independence study shows that the 
bias-corrected estimator of mutual information is centred near zero under 
independence, whereas the uncorrected estimator is positive.

\medskip
\noindent\textbf{Future work.}
Several extensions are natural:
\begin{itemize}
\item Extending the univariate reduction principle and the asymptotic 
equivalence to the trivariate case (see our companion paper \cite{Ba2026trivar}).
\item Developing bootstrap-based variance estimators for improved 
finite-sample inference on the bias-corrected estimators.
\item Applying the framework to dependent data (time series, spatial data, 
Markov chains), where the i.i.d.\ assumption is relaxed.
\item Investigating the interplay between bias correction and model selection 
criteria based on information measures (AIC, BIC, MDL).
\item Applying the framework to other actuarial contexts: health insurance, 
life insurance, credit risk, and reinsurance.
\end{itemize}

\section*{Acknowledgements}
I wish to acknowledge Professor Gane Samb Lo, my former PhD supervisor, whose 
guidance marked the beginning of my journey into research. I remain deeply 
grateful for his mentorship and for the lasting influence of his guidance on 
my work.

\end{document}